\documentclass[12pt]{amsart}
\RequirePackage{amsmath,amsthm,amssymb,amsfonts}
\title[Derivatives at the Boundary]{Derivatives at the Boundary
for Analytic Lipschitz Functions}
\author{Anthony G. O'Farrell}
\email{anthony.ofarrell@nuim.ie}
\address{Mathematics and Statistics, NUI, Maynooth, Co. Kildare, Ireland}
\date{\today}
\dedicatory{}
\def\Label{\label}

\newtheorem{theorem}{Theorem}[section]
\newtheorem{corollary}[theorem]{Corollary}
\newtheorem{lemma}[theorem]{Lemma}

\newcommand{\B}{\mathbb{B}}
\newcommand{\C}{\mathbb{C}}
\newcommand{\N}{\mathbb{N}}
\newcommand{\R}{\mathbb{R}}

\newcommand{\bdy}{\textup{bdy}}
\newcommand{\clos}{\textup{clos}}

\newcommand{\dist}{\textup{dist}}

\newcommand{\lip}{\textup{lip}}

\newcommand{\half}{\frac12}
\newcommand{\dsty}{\displaystyle}

\def\sideremark#1{\ifvmode\leavevmode\fi\vadjust{
\vbox to0pt{\hbox to 0pt{\hskip\hsize\hskip1em
\vbox{\hsize1cm\tiny\raggedright\pretolerance10000
\noindent #1\hfill}\hss}\vbox to8pt{\vfil}\vss}}}

\begin{document}
\begin{abstract}
We consider the behaviour of holomorphic functions 
on a bounded open subset of the plane,
satisfying a Lipschitz condition with exponent
$\alpha$, with $0<\alpha<1$, in the vicinity of
an exceptional boundary point where all such functions exhibit some
kind of smoothness. Specifically, we consider the
relation between the abstract idea of a bounded
point derivation on the algebra of
such functions and the classical complex derivative
evaluated as a limit of difference quotients.
We show that whenever such a bounded point derivation
exists at a boundary point $b$, 
it may be evaluated by taking a limit
of classical difference quotients, for approach
from a set having full area density at $b$.
\end{abstract}

\maketitle

\section{Introduction}
In a first paper \cite{bsaf} on the subject,
we confined attention
to the situation in which the boundary point
is nicely accessible from $U$.  
This paper reports on a continuation of the investigation
described in \cite{bsaf}, and we refer the reader to
that paper for background and notation.

We recall for emphasis that the main object of interest
is the algebra $A=A_\alpha(U)$ of $\lip\alpha$ 
(note the \lq\lq little-lip")
functions on the bounded open set $U\subset\C$,
and that $0<\alpha<1$. 

We note one necessary condition that follows from the theorem
of \cite{LO}: If
there exists a 
nonzero continuous
point derivation on $A=A_\alpha(U)$ at a given 
boundary point $b$,
then $U$ has full area density at $b$,
i.e.
$$ \lim_{r\downarrow0}\frac{\mathcal{L}^2(U\cap\mathbb{B}(b,r))}{\pi r^2}
=1,$$
where 
$\mathbb{B}(b,r)$ denotes the closed ball of radius $r$.  Here
we use the notation that $\mathcal{L}^d$ denotes 
$d$-dimensional Lebesgue measure.

We say that a sequence $(z_n)_n$ of points of $U$
{\em converges non-tangentia\-lly to }$b$, written
$z_n\to_\textup{nt} b$, if there exists a constant
$t>0$ such that 
$$\dist(z_n,\C\setminus U)\ge t|z_n-b|,\ \forall n.$$

Our main result in \cite{bsaf}\footnote{
There is a misprint in Equation (3), page 142 of this paper:
$\alpha$ should be $1-\alpha$.  Also, on page
144, line 5, it should say $R_a-L_a=(1-\hat{T}(a))R_a$.
With these corrected, the
argument is valid and the result stands.
} was the following:

\begin{theorem}
Let $0<\alpha<1$, let $U\subset\C$ be a bounded open set,
let $b\in\bdy(U)$, $z_n\in U$, $z_n\to_\textup{nt} b$.
Suppose $A=A_\alpha(U)$ admits a nonzero continuous point derivation
at $b$.  Let $\partial$ be the normalised
derivation at $b$. 
Then  for each $f\in A$, we have
$$ \frac{f(z_n)-f(b)}{z_n-b} \to \partial f.$$
\end{theorem}

Now we consider the general situation, in which there
may not exist any sequence from
$U$ that approaches $b$ nontangentially.
Without going into details
about Hausdorff content, we remark that for a closed ball
$M_*^\beta(\B(a,r))=r^{\beta}$,
and also that $M_*^\beta$
is countably subadditive. As a result, it is easy to construct
many examples $U$ in which the complement of $U$ is a countable union
of  closed balls, and the singleton $\{b\}$, and in which
$A$ has a continuous point derivation at $b$,
even though there is no sequence $z_n\in U$
approaching $b$ nontangentially.  
All you have to do is make sure that the sum of
the $(1+\alpha)$-th powers of
the radii of all the closed balls that meet
$A_n(b)$ is no greater than $s_n/4^n$, where $\sum_ns_n<+\infty$.

Here is our main result:

\begin{theorem}\Label{T:main}
Let $0<\alpha<1$, let $U\subset\C$ be a bounded open set,
and let $b\in\bdy(U)$.
Suppose $A=A_\alpha(U)$ admits a nonzero continuous point derivation
at $b$.  Let $\partial$ be the normalised
derivation at $b$. Then there exists a set $E\subset U$, 
having full area density at $b$, such that
for each $f\in A$, we have
$$ \lim_{E\ni z\to a}\frac{f(z)-f(b)}{z-b} = \partial f.$$
\end{theorem}

As with the earlier theorem, 
we would expect that this result could be extended
to higher order derivations and to weak-star
continuous derivations on the big Lip algebra.


This result allows us to give a number of conditions
equivalent to the existence of a bounded point derivation
at $b$:

\begin{corollary}\Label{C:main}
Let $0<\alpha<1$, let $U\subset\C$ be a bounded open set,
and let $b\in\bdy(U)$.
Then the following four conditions are equivalent:
\\
(1) 
$A=A_\alpha(U)$ admits a nonzero continuous point derivation
at $b$.  
\\
(2) There exist $a_n\in U$, $a_n\to b$, such that
$\dsty
\frac{f(a_n)-f(b)}{a_n-b}$ 
is bounded for each $f\in A$.
\\
(3) There exist $a_n\in U$, $a_n\to b$, such that
$\dsty
\frac{f(a_n)-f(b)}{a_n-b}$ 
converges for each $f\in A$.
\\
(4) There exists $E\subset U$ having full area density at $b$
such that 
\\$\dsty
\lim_{E\ni a\to b}\frac{f(a)-f(b)}{a-b}$ 
exists for each $f\in A$.
\end{corollary}

In the paper \cite{bsaf} we introduced a number
of new ideas that allow the use of arguments
similar to those used in the theory of uniform 
algebras in this present context of Lipschitz algebras
with exponent less than $1$.  Essentially, the
main idea is that $\lip\alpha$, for $0<\alpha<1$
is \lq\lq sufficiently close" to $C^0$, so that
the elements of its dual space, although they are 
not measures, are \lq\lq sufficiently similar" 
to measures so that methods that work for measures 
can also be made to work for them.  
We are going to use these same ideas, supplemented
by some observations drawn from potential
theory, to prove our main theorem.

The paper is organised as follows: 
In Section \ref{S:ext+cor} we make a few preliminary remarks and
establish that Corollary \ref{C:main} will follow from
the main theorem.
In Section \ref{S:tools-PT}, we establish a few lemmas
from potential theory. These real-variable results about
Riesz potentials hold not just in two dimensions,
and are presented here in the context of $d$-dimensional
Euclidean space.
In section \ref{S:tools-CT} we recall the tools involving
representing measures and the Cauchy transform that
were introduced in \cite{bsaf}, and 
we apply the lemmas from potential theory
to derive estimates for Cauchy transforms. In
Section \ref{S:proof} we complete the proof of the 
main theorem.

\section{Preliminaries}\Label{S:ext+cor}

\subsection{Standing Conventions in this Paper}\Label{SS:conventions}
Throughout the paper, $0<\alpha<1$, $U\subset\C$ is a bounded
open set, $A=A_\alpha(U)$, $Y=\clos(U)$, and $b\in X=\bdy(U)$.
By \emph{measure} we shall mean a Borel-regular
complex measure $\mu$ such that its total-variation
measure is a Radon measure, i.e. assigns finite measure to
each compact set.  

\subsection{Extensions}
Let $\lip\alpha$ denote, for short, the global space
$\lip(\alpha,\C)$ of bounded $\lip\alpha$
functions on $\C$.  This space forms a Banach algebra
under pointwise operations. Each $f\in A$ may be extended
(in many different ways) to an element of 
$\lip\alpha$, without increasing its pure
seminorm $\|f\|'=\kappa(f)$ or supremum (cf. \cite{bsaf}). 
Thus the restriction
map to $U$ (or $Y$, or $X$) makes $A$ isometric
to a quotient algebra of 
$$ \tilde A = \{ f\in \lip\alpha: 
f \textup{ is holomorphic on }U\}.$$
We shall find it convenient to work
with globally-defined functions in the sequel.

Let us denote
$$ \mathcal{A}:= \{f\in\tilde A: f\textup{ is holomorphic near }b\}.$$
Then \cite{LO} $\mathcal{A}$ is a dense subalgebra of $\tilde A$,
and the restriction algebra $\mathcal{A}|U$ is a dense subalgebra
of $A$. As a consequence, the set of functions 
$f\in A$ that extend holomorphically to a neighbourhood of $b$
is dense in $A$.

\subsection{Proof of Corollary \ref{C:main}}
\begin{proof}
Theorem \ref{T:main} tells us that condition (1) in the
corollary implies condition (4). It is obvious that
(4)$\implies$(3)$\implies$(2), so we just have to see that
(2)$\implies$(1).

Suppose (2) holds. Then by the Banach-Steinhaus Theorem
there is a constant $M>0$ such that
$$
\left|\frac{f(a_n)-f(b)}{a_n-b}\right| \le M\|f\|_A$$ 
for each $f\in A$.  Hence if $f\in A$ extends holomorphically
to a neighbourhood of $b$, we have $|f'(b)|\le M\|f\|_A$.
Since the set of such functions is dense in $A$,
it follows that the functional $f\mapsto f'(b)$
has a unique continuous extension $\partial:A\to\C$,
and clearly $\partial$ is a point derivation at $b$. 
\end{proof}

\section{Tools from Potential Theory}\Label{S:tools-PT}
\subsection{The Riesz Capacities}
Fix $d\in\N$.  Let $0<s<d$. For compact $K\subset\R^d$ we define
$C_s(K)$, the {\em order $s$ Riesz capacity of the set $K$}, to be
the supremum of the total variations $\|\mu\|$, where $\mu$ 
ranges over all positive Radon measures that are supported on $K$
and have $\displaystyle \frac1{|x|^s}*\mu\le1$ on $\R^d$, i.e.
the Riesz potential
$$ \int_K \frac{d\mu(x)}{|x-y|^s} \le 1,\ \forall y\in\R^d.$$
For general sets $E\subset\R^d$, we then define
$C_s(E)$ to be the inner capacity
$$ \sup_{K\subset E, K\textup{ compact}}
C_s(K).$$
For more about these capacities, see \cite{Fed,Miz}.
We note from \cite{Fed}, 
for future use, the following facts about $C_s$:
\begin{enumerate}
\item $C_s(\mathbb{B}(a,r))\le r^s$ for all $a\in\R^d$ and $r>0$.
\item There is a constant $c$, depending only on $d$, such that
\\
$C_s(\mathbb{B}(a,r))\ge c\cdot r^s$ for all $a\in\R^d$ and $r>0$.
\item More generally, 
$C_s(E)\ge c\cdot \mathcal{L}^d(E)^{\frac sd}$ for all $E\subset\R^d$.
\item $C_s$ is countably subadditive.
\end{enumerate}

\subsection{Estimate for the potential}
We need a couple of estimates involving these capacities
and potentials.  The first is the following:

\begin{lemma}\Label{L:Riesz-1} Let $b\in\R^d$, $\mu$ be a positive Borel-regular
measure on $\R^d$
having compact support and no mass at $b$, let $0<s<d$, $\epsilon>0$,
and
$$ E = \left\{ a\in\R^d: |a-b|^s\cdot \int_{\R^d}
\frac{d\mu(x)}{|x-a|^s} \ge \epsilon\right\}.$$
Then
$$ \sum_{n=1}^\infty 2^{sn}C_s(A_n(b)\cap E) <+\infty.$$
\end{lemma}

Here, as usual, $A_n(b)$ denotes an annulus:
$$ A_n(b) = \left\{ x\in\R^d: \frac1{2^{n+1}}\le |x-b| \le \frac1{2^n} 
\right\}.
$$

\begin{corollary}\Label{C:Riesz-1} Under the same assumptions
on $b$,$\mu$,$s$,$\delta$ and $E$, the set $E$ has Lebesgue density zero at 
$b$, i.e.
$$ \lim_{r\downarrow0}\frac{\mathcal{L}^d(E\cap\mathbb{B}(b,r))}{r^d}
=0.$$
\end{corollary}

This lemma is well-known to potential-theorists,
and may be found in the book of Mizuta \cite[Theorem 5.3]{Miz}.
(One should remark that the case $(d,s)=(2,1)$ of the 
 corollary
is due to 
Browder and was used in the proof (1967) of his Metric Density
Theorem \cite[Theorem 3.3.9, p.177]{B}, and that the 
same case of the full theorem was known in 1974
\cite{O}.  The case
$d=2$, $1<s<2$ was given in \cite[page 421]{O2}.
The case $(d,s)=(d,d-2)$ is in
the book of Armitage and Gardiner \cite[Theorem 7.7.2]{Gar}.)

The Corollary is immediate from the fact (cited above) that
the capacity $C_s(E)$ is bounded below by a 
constant multiple of $\mathcal{L}^d(E)^{\frac sd}$.

\subsection{Estimate for the double-layer potential}
The second estimate involves a double-layer potential:
\begin{lemma}\Label{L:Riesz-2} Let $b\in\R^d$, $\mu$ 
be a positive measure on $\R^d\times\R^d$
having compact support and no mass on
$\{b\}\times\R^d$ or on $\R^d\times\{b\}$, let $s>0$, $t>0$, 
$0<s+t<d$, $\epsilon>0$,
and
$$ E = \left\{ a\in\R^d: |a-b|^{s+t}\cdot \int_{\R^d\times\R^d}
\frac{d\mu(x,y)}{|x-a|^s|y-a|^t} \ge \epsilon\right\}.$$
Then
$$ \sum_{n=1}^\infty 2^{(s+t)n}C_{s+t}(A_n(b)\cap E) <+\infty.$$
\end{lemma}
 
This seems to be new, and we shall provide a proof.

We remark that the integral is finite for almost all
$a$ 
with respect to area measure, as may be seen by applying Fubini's
theorem.

\begin{proof}[Proof of Lemma \ref{L:Riesz-2}]
Suppose that, on the contrary,
$$ \sum_{n=1}^\infty 2^{(s+t)n}C_{s+t}(A_n(b)\cap E) =+\infty.$$

Fix a positive number $M$ greater than the diameter
of the support of $\mu$ and greater than the distance from
$(b,b)$ to any point of the support of $\mu$. 

Fix $N\in\N$. Since $C_{s+t}(A_n(b)\cap E)$ does not exceed
$2^{-(s+t)n}$, we may choose $M\ge N$ (depending on $N$)
such that
$$ 2\le \sum_{n=N}^M 2^{(s+t)n}C_{s+t}(A_n(b)\cap E) \le3.$$
For each $n$ from $N$ to $M$, choose a positive measure $\nu_n$
supported on a compact subset of $A_n(b)\cap E$, with
$\nu_n*\frac1{|x|^s}\le1$ and $\|\nu_n\|\ge \half\cdot C_{s+t}(A_n(b)\cap E)
$.
Let $\lambda_N:= \sum_N^M 2^{(s+t)n}\nu_n$.
Then $\lambda_N$ is supported in $\B(b,2^{-N})$
and has $1\le\|\lambda_N\|\le3$. Also
$$\begin{array}{rcl}
\delta &\le&
\dsty\int|a-b|^{s+t}\int\frac{
d\mu(x,y)}{
|x-a|^s|y-a|^t}
d\lambda_N(a)\\
&=&\dsty\int
G_N(x,y)d\mu(x,y),
\end{array}
$$
where
$$ G_N(x,y):= 
\int \frac{|a-b|^{s+t}}{
|x-a|^s|y-a|^t}
d\lambda_N(a).
$$
We claim (1) that $G_N(x,y)$ is bounded above, independently of
$N$, $x$ and $y$, for all $(x,y)$ in the support of
$\mu$ with
$$(x,y)\not\in \left(\{b\}\times R^d\right) \cup \left(\R^d\times\{b\}\right),$$
and (2) that $G_N(x,y)\to 0$ pointwise there as $N\uparrow\infty$.
Once we have this claim, it will follow from the Lebesgue Dominated 
Convergence Theorem that $\delta\le0$, a contradiction,
and the proof will be complete.

\medskip
To prove the claims, fix $x\not=b$ and $y\not=b$
with $(x,y)$ in the support of $\mu$.

For large $N$, and $a$ belonging to the support of
$\lambda_N$, we have that $|x-a|$ is comparable to $|x-b|$
and $|y-a|$ is comparable to $|y-b|$, so $G_N(x,y)$
is bounded above by twice
$$\frac1{|x-b|^s\cdot|y-b|^t}
\int|a-b|^{s+t}d\lambda_N(a) 
\le
\frac{2^{-(s+t)N}\|\lambda_N\|}{|x-b|^s\cdot|y-b|^t},
$$
which tends to zero as $N\uparrow\infty$.
This proves claim (2).

To get a bound for $G_n(x,y)$, choose  $p\in\N$ and $q\in\N$
with $x\in A_p(b)$ and $y\in A_q(b)$.

Fix $n$, with $N\le n\le M$, and 
fix $a$ in the support of $\nu_n$.

Consider the following cases, which together cover all possibilities:

\begin{enumerate}
\item\label{Case1}$|p-n|>1$ and $|q-n|>1$.
\item\label{Case2}$|p-n|>1$ and $|q-n|\le1$.
\item\label{Case3}$|p-n|\le1$ and $|q-n|>1$.
\item\label{Case4}$|p-n|\le1$ and $|q-n|\le1$.
\end{enumerate}
Observe that for given $x$ and $y$, there will be at most
$3$ values of $n$ for which Case \ref{Case2} holds,
and that the same is true for Cases \ref{Case3} and \ref{Case4}.

\subsubsection{Case (\ref{Case1})}
We have $|x-a|\ge2^{-(n+2)}$, $|y-a|\ge2^{-(n+2)}$, $|a-b|\le2^{-n}$,
so
$$
2^{(s+t)n}\int\frac{|a-b|^{s+t}}{
|x-a|^s\cdot|y-a|^t}
d\nu_n(a)
\le 4^{s+t}\cdot2^{(s+t)n}\cdot\|\nu_n\|.
$$
\subsubsection{Case (\ref{Case2})} We have
$$\begin{array}{rcl}
\dsty
&&2^{(s+t)n}\dsty\int\frac{|a-b|^{s+t}}{
|x-a|^s\cdot|y-a|^t}
d\nu_n(a)\\
&\le&
\dsty4^s\cdot\int\frac{d\nu_n(a)}{
|y-a|^t}
\\
&\le&
\dsty4^s\cdot(M+1)^s\cdot\int\frac{d\nu_n(a)}{
|y-a|^{s+t}} \le 4^s\cdot(M+1)^s,
\end{array}
$$
since $|y-a|\le |y-b|+|b-a|\le M+1$. 

\subsubsection{Case (\ref{Case3})} This is similar to Case (\ref{Case2}),
and we get the bound $4^t\cdot(M+1)^t$.

\subsubsection{Case (\ref{Case4})}
If $|y-a|\le|x-a|$, then
$$ \frac1{|x-a|^s\cdot|y-a|^t} \le \frac1{|y-a|^{s+t}},$$
and if $|x-a|\le|y-a|$, then
$$ \frac1{|x-a|^s\cdot|y-a|^t} \le \frac1{|x-a|^{s+t}},$$
so in either case
$$ \frac1{|x-a|^s\cdot|y-a|^t} \le 
\max\left\{
\frac1{|x-a|^{s+t}},\frac1{|y-a|^{s+t}}
\right\},$$
hence
$$\begin{array}{rcl}
\dsty
&&2^{(s+t)n}\dsty\int\frac{|a-b|^{s+t}}{
|x-a|^s\cdot|y-a|^t}
d\nu_n(a)\\
&\le&
\dsty\max\left\{
\int\frac{d\nu_n(a)}{|x-a|^{s+t}},
\int\frac{d\nu_n(a)}{|y-a|^{s+t}}
\right\}
\le1.
\end{array}
$$

Thus, combining all these estimates, we get
$$ 
\begin{array}{rcl}
G_N(x,y)&\le& 
\dsty4^{s+t}\sum_N^M2^{(s+t)n}\|\nu_n\| + 6\cdot(4M+4)^{\max\{s,t\}}+3\cdot1 
\\
&\le& 12\cdot(4M+4)^{s+t}.
\end{array}
$$
This proves claim (1) and concludes the proof
of the lemma.
\end{proof}

Again we obtain:
\begin{corollary}\Label{C:Riesz-2} Under the same assumptions
on $b$,$\mu$,$s$,$\delta$ and $E$, the set $E$ has Lebesgue 
density zero at $b$.
\end{corollary}

We expect that similar lemmas could be proved about
triple-layer potentials, and so on.

\subsection{A Refined Estimate}
We can do better if we consider potentials with
a factor $|x-y|^u$ (with $0<u<1$) on the measure:

\begin{lemma}\Label{L:Riesz-3} Let $b\in\R^d$, $\mu$ 
be a positive measure on $\R^d\times\R^d$
having compact support and no mass on
$\{b\}\times\R^d$ or on $\R^d\times\{b\}$.
Let $s>0$, $t>0$, $0<u<\min\{1,s,t\}$ and $s+t-u<d$.
Let $\epsilon>0$,
and
$$ E := \left\{ a\in\R^d: |a-b|^{s+t-u}\cdot \int_{\R^d\times\R^d}
\frac{|x-y|^ud\mu(x,y)}{|x-a|^s|y-a|^t} \ge \epsilon\right\}.$$
Then
$$ \sum_{n=1}^\infty 2^{(s+t-u)n}C_{s+t-u}(A_n(b)\cap E) <+\infty.$$
\end{lemma}

\begin{proof}
Since $0<u<1$, we have $|x-y|^u\le|x-a|^u+|y-a|^u$, so
$$
\frac{|x-y|^u}{|x-a|^s|y-a|^t} 
\le \frac{1}{|x-a|^{s-u}|y-a|^t} 
+
\frac{1}{|x-a|^s|y-a|^{t-u}}.
$$ 
Thus $E_1\cap E_2\subset E$, where
$$ 
E_1 := \left\{ a\in\R^d: |a-b|^{s+t-u}\cdot \int_{\R^d\times\R^d}
\frac{d\mu(x,y)}{|x-a|^{s-u}|y-a|^t} \ge \epsilon/2\right\}
$$
and
$$
 E_2 := \left\{ a\in\R^d: |a-b|^{s+t-u}\cdot \int_{\R^d\times\R^d}
\frac{d\mu(x,y)}{|x-a|^s|y-a|^{t-u}} \ge \epsilon/2\right\}.
$$
Applying Lemma \ref{L:Riesz-2} twice, we see that $E_1$ and
$E_2$ satisfy
$$ \sum_{n=1}^\infty 2^{(s+t-u)n}C_{s+t-u}(A_n(b)\cap E_j) <+\infty,$$
so the result follows from the subadditivity of $C_{s+t-u}$.
\end{proof}

Again we obtain:
\begin{corollary}\Label{C:Riesz-3} Under the same assumptions
on $b$,$\mu$,$s$,$\delta$ and $E$, the set $E$ has Lebesgue 
density zero at $b$.
\end{corollary}

\subsection{A Density Lemma}
In the next lemma, we remove a restriction on $\mu$:
the measure is allowed
to have mass on the horizontal and vertical slices 
through the point $b$, as long as it has no mass
at the point $(b,b)$.  The conclusion is not the convergence
of a Wiener-type series, but that an exceptional set
has Lebesgue density zero. 

\begin{lemma}\Label{L:Riesz-4}
Let $b\in\R^d$, and $\mu$ 
be a positive measure on $\R^d\times\R^d$
having compact support and no mass at the
point $(b,b)$. 
Let $s>0$, $t>0$, $0<u<\min\{1,s,t\}$ and $s+t-u<d$.
Let $\epsilon>0$,
and
$$ E := \left\{ a\in\R^d: |a-b|^{s+t-u}\cdot \int_{\R^d\times\R^d}
\frac{|x-y|^ud\mu(x,y)}{|x-a|^s|y-a|^t} \ge \epsilon\right\}.$$
Then
$E$ has Lebesgue density zero at $b$.
\end{lemma}

For convenience, let us say that a set $E$ is
{\em $s$-thin at $b$} if 
$$ \sum_{n=1}^\infty 2^{sn}C_{s}(A_n(b)\cap E) <+\infty.$$
As mentioned already, if $E$ is $s$-thin at $b$ for some
$s$ with $0<s<d$, then $E$ has Lebesgue density 
zero at $b$. So if $E\subset E_1\cup\cdots\cup E_n$
and for each $j$ there is some $s_j\in(0,d)$ such that
$E_j$ is $s_j$-thin, then $E$ has Lebesgue density
zero at $b$.

\begin{proof}
Let us denote the double-layer potential in the statement by $P(\mu)$: 
$$ P(\mu)(a):=
\int_{\R^d\times\R^d}
\frac{|x-y|^ud\mu(x,y)}{|x-a|^s|y-a|^t}.$$
Write $\mu=\mu_1+\mu_2+\mu_3$, where 
$\mu_2$
is the restriction of $\mu$ to $V:= \{b\}\times\R^d$
(i.e. $\mu_2(T)=\mu(T\cap V)$ for all sets $T$),
$\mu_3$
is the restriction of $\mu$ to $H:= \R^d\times\{b\}$,
and $\mu_1$
is the restriction of $\mu$ to the complement of $H\cup V$.
Then we may write
$P(\mu)  = P_1(a)+P_2(a)+P_3(a)$, where $P_j:= P(\mu_j)$.

Note that none of the $\mu_j$ has a point mass at $(b,b)$
and that $\mu_1$ has no mass on $V\cup H$.

Let
$$ E_j:= \left\{
a:
|a-b|^{s+t-u}\cdot P_j(a) \ge \delta/3
\right\}
$$
for $j=1$,$2$,$3$.  Then
$E\subset E_1\cup E_2\cup E_3$, so it suffices to show that
each $E_j$ has Lebesgue density zero at $b$.

\medskip
Lemma \ref{L:Riesz-3} applies, with $\mu$ replaced by $\mu_1$,
and tells us that $E_1$ is $(s+t-u)$-thin at $b$,
so the case $j=1$ is done.

\medskip
Consider the case $j=2$.

Using the fact that $t\mapsto t^u$ is subadditive, as in
the proof of Lemma \ref{L:Riesz-3}, we see that
$P_2(a)\le P_{21}(a) + P_{22}(a)$, where
$$ \begin{array}{rcl}
P_{21}(a) &=& \dsty
\int
\frac{d\mu(x,y)}{
|x-a|^{s-u}\cdot|y-a|^{t}}
\\
P_{22}(a) &=& \dsty
\int
\frac{d\mu(x,y)}{
|x-a|^s\cdot|y-a|^{t-u}}.
\end{array}
$$

Let
$$ E_{2k}:= \left\{
a:
|a-b|^{s+t-u}\cdot P_{2k}(a) \ge \delta/6
\right\}
$$
for $k=1$,$2$.  Then 
$ E_2\subset E_{21}\cup E_{22}$, so it
suffices to show that both 
$E_{2k}$ have Lebesgue density zero at $b$.

\smallskip
Consider 
$E_{21}$.  The measure $\mu_2$
is supported on $V=\{b\}\times\R^d$, so
it is the product of the unit point mass at $b$
and a measure $\mu_{22}$ on $\R^d$, and we have
$$ P_{21}(a) = \frac1{|b-a|^{s-u}}\cdot
\int_{\R^d} 
\frac{d\mu_{22}(y)}{
|y-a|^t},
$$
so 
$$ E_{21} = \left\{
a: 
|a-b|^t\cdot
\int_{\R^d} 
\frac{d\mu_{22}(y)}{
|y-a|^t} \ge \delta/6
\right\}.
$$
Since $\mu_2$ has no point mass at $(b,b)$,
it follows that $\mu_{21}$ has no point mass
at $b$.
Thus we may apply Lemma \ref{L:Riesz-1},
with $s$ replaced by $t$, to conclude that $E_{21}$
is $t$-thin at $b$.

Next, consider 
$E_{22}$.  
We have
$$ P_{22}(a) = \frac1{|b-a|^s}\cdot
\int_{\R^d} 
\frac{d\mu_{22}(y)}{
|y-a|^{t-u}},
$$
so 
$$ E_{22} = \left\{
a: 
|a-b|^{t-u}\cdot
\int_{\R^d} 
\frac{d\mu_{22}(y)}{
|y-a|^{t-u}} \ge \delta/6
\right\}.
$$
So we may apply Lemma \ref{L:Riesz-1},
with $s$ replaced by $t-u$, to conclude that $E_{22}$
is $(t-u)$-thin at $b$.

Thus $E_2$ has density zero at $b$. 

\medskip
A completely parallel argument shows that $E_3$
is contained in the union of an $s$-thin set and
an $(s-u)$-thin set, and hence has density zero at $b$.

The proof is complete.
\end{proof}

The above argument does not work if we try to
relax the condition $u<\min\{1,s,t\}$.  
In the sequel we shall encounter a case
with $d=2$ and $s=t=u=1$, so we record
a substitute result:
\begin{lemma}\Label{L:Riesz-5}
Let $b\in\R^d$, and $\mu$ 
be a positive measure on $\R^d\times\R^d$
having compact support and no mass at the
point $(b,b)$. 
Let $s>0$, $t>0$, $0<u\le\min\{1,s,t\}$ and $s+t-u<d$.
Then there exists a positive constant $M$
such that the set
$$ E := \left\{ a\in\R^d: |a-b|^{s+t-u}\cdot \int_{\R^d\times\R^d}
\frac{|x-y|^ud\mu(x,y)}{|x-a|^s|y-a|^t} \ge M\right\}.$$
has Lebesgue density zero at $b$.
\end{lemma}

\begin{proof} 
It remains only to consider the cases $u=1$, $u=s$ and $u=t$.

We can follow the same proof as in 
Lemma \ref{L:Riesz-4}, with $\delta$ replaced by $M$,
and most of it works for any positive $M$.

The case $u=1$ causes no problem. 

When $u=s$ or $u=t$ and we consider $E_1$, we may
have to apply Lemma \ref{L:Riesz-1} instead
of Lemma \ref{L:Riesz-2}, as the integrand
depends on only one of the variables $x$
and $y$.   

When we come to consider $E_{22}$, and now suppose that
$u=t$, we find that
$$ P_{22}(a) = \frac1{|b-a|^s}\cdot\|\mu\| $$
so 
$$ E_{22} = \left\{
a: \|\mu\| \ge M/6
\right\}.
$$
Thus $E_{22}$ is either empty or $\R^d$, depending
on the value of $M$.  If we take $M=6\|\mu\|$
(or greater), then $E_{22}=\emptyset$, which surely has 
density zero at $b$.
\end{proof}

\section{Dual Spaces and Distributions}
\Label{S:tools-CT}
Recall our standing assumptions from Subsection \ref{SS:conventions}.
In particular, throughout this section we use the facts that
$0<\alpha<1$,  $X$ is a compact subset of $\C$ and $b\in X$.

\subsection{A Construction}
If $\mu$ is any complex measure
on $X\times X$, having no mass on the diagonal,
we define 
\begin{equation}\label{E:2}
L(\mu)(f) :=
 \int_{X\times X} 
\frac{f(z)-f(w)}{|z-w|^\alpha}
\,d\mu(z,w),\ \forall f\in\lip\alpha.
\end{equation}
Then $L(\mu)$ is an element of $(\lip\alpha)^*$,
the dual of $\lip\alpha$, of norm at most $\|\mu\|$.
(This construction is classical.)

The linear map $\mu\mapsto L(\mu)$ is not injective,
because of the antisymmetry in the integrand. If we
define $R(z,w):= (w,z)$, then
$L(R_\sharp\mu) = -L(\mu)$, where $R_\sharp\mu$
is the push-forward measure, defined by
$$ R_\sharp\mu(E) = \mu(R^{-1}(E)),\ \forall E\subset X\times X.$$
We record here, for future use, a consequence
of this remark:

\begin{lemma}\Label{L:avoid}
Let $b\in X$.
Let $\mu$ be a complex measure on $X\times X$,
having no mass on the diagonal. Then
there exists a measure $\mu'$ on $X\times X$ such that
$\mu'$ has no mass on the vertical slice $V=\{b\}\times X$ and
$L(\mu)=L(\mu')$.
\end{lemma}
\begin{proof}
Write $\mu=\mu_1+\mu_2$, where $\mu_2$ is the restriction of
$\mu$ to $V$.  Then $\mu_1$ has no mass on $V$,
and $\mu_2$ has no mass at $(b,b)$.  Take
$\mu'=\mu_1 - R_\sharp\mu_2$. Then since
$R_\sharp\mu_2$ is supported on the horizontal
slice $H=X\times\{b\}$ and has no mass at $(b,b)$,
it has no mass on $V$. Also, $L(\mu')=L(\mu)$.
\end{proof}. 

\subsection{Distributions}
Let
$\mathcal D$ denote the space of test functions
(i.e. $C^\infty$ functions having compact support),
and let $\mathcal D'$ denote its dual, the Schwartz
distribution space.

The restriction of $L(\mu)$ to $\mathcal{D}$
is a distribution. We denote it by the same symbol
$L=L(\mu)$.
This distribution will not, in general,
be representable by integration against 
a locally-integrable function
or a measure. 

Since $\langle\phi,L\rangle$ is unaffected if $\phi$ is
altered away from $X$, it is clear that
$L$ has support in $X$. Thus we can also define
$\langle\phi,L\rangle$ for any function $\phi$
defined and $C^\infty$ on a neighbourhood of $X$
to be 
$\langle\tilde\phi,L\rangle$ where $\tilde\phi$
is any element of $\mathcal E$ (the space of
globally-defined $C^\infty$ functions) that agrees
with $\phi$ near $X$.   For instance, 
$\displaystyle \left\langle \frac1{z-a},L\right\rangle $
makes sense, for $a\not\in X$. Similarly, $\langle f,L\rangle$
makes sense whenever $f$ is defined on some neighbourhood
of $X$ and satisfies a little-lip$\alpha$ condition there. 

\subsection{Cauchy Transforms} 
The Cauchy transform of $\phi\in\mathcal D$
is its convolution
$$ \hat\phi := \phi*\left(
\frac1{\pi z}
\right) $$
with the fundamental solution of $\displaystyle 
\frac{\partial}{\partial\bar z}$. In other words,
$$ \hat\phi(z) = \frac1\pi \int\frac{\phi(\zeta)}{z-\zeta} 
d\mathcal{L}^2(\zeta),$$
for all $z\in\C$.
This function belongs to the space $\mathcal E$,
and satisfies
$$ \frac{\partial \hat\phi}{\partial z} = \phi. $$

For distributions $T$ having compact support, we define
$$ \langle\phi,\hat T\rangle =
-\langle\hat\phi, T\rangle, \forall\phi\in\mathcal D.$$

To analyse the transform in case $T=L(\mu)$,
consider the function
$$ H(a):= H(\mu)(a) := \frac1\pi \int
\frac{z-w}{(z-a)(w-a)|z-w|^\alpha}
\,d\mu(z,w),$$
for $a\in \C$.  This is well-defined whenever
$$ \widetilde{H}(a):= \widetilde{H}(\mu)(a) := \int
\frac{|z-w|^{1-\alpha}}{|z-a|\cdot|w-a|}
\,d|\mu|(z,w)<\infty,$$
which happens almost everywhere with respect to area measure,
and $\tilde H$ is locally-integrable,
as is seen by an application of Fubini's Theorem.
Also $|H(a)|\le \tilde{H}(a)$ for all such $a$.
Another Fubini calculation yields
$$ \langle\phi, \widehat{L(\mu)}\rangle = 
\int_\C \phi\cdot H d\mathcal{L}^2, $$
for all $\phi\in\mathcal D$.  Thus
$H$ represents $\widehat{L}$.  Based on this, 
we usually write $\widehat{L}(a)$ for $H(a)$.
Note that
$$ \widehat{L}(a) = H(a)=\left\langle\frac1{\pi(a-z)},L\right\rangle,$$
whenever $a\not\in X$.

\subsection{Estimates for the Cauchy Transform}
First, an estimate for $\widetilde{H}(\mu)$.
Notice that $\widetilde{H}(\mu)$ depends only on
the total variation measure $|\mu|$.
\begin{lemma}\Label{L:Newton-1}
Let the complex measure $\mu$
have support in $X\times X$ and no mass at the point
$(b,b)$. 
Let
$\delta>0$ be given. Then the set
$$ E:= \left\{
a:
|a-b|^{1+\alpha}\cdot \widetilde{H}(\mu)(a) < \delta
\right\}
$$
has full area density at $b$.
\end{lemma}
\begin{proof}
Just apply Lemma \ref{L:Riesz-4} with $d=2$, $\mu=|\mu|$,
$s=t=1$ and $u=1-\alpha$.
\end{proof}

Next, an estimate for $\widetilde{H}(|z-w|^\alpha\cdot\mu)$
(by $|z-w|^\alpha\cdot\mu$ we mean the measure
obtained by multiplying $\mu$ by the function
$(z,w)\mapsto |z-w|^\alpha$).
\begin{lemma}\Label{L:Newton-2}
Let the complex measure $\mu$
have support in $X\times X$ and no mass at the point
$(b,b)$. 
Let
$\delta>0$ be given. Then there exists a 
constant $M>0$ such that the set
$$ E:= \left\{
a:
|a-b|\cdot \widetilde{H}(|z-w|^\alpha\cdot \mu)(a) < M
\right\}
$$
has full area density at $b$.
\end{lemma}
\begin{proof}
Explicitly,
$$ \widetilde{H}(|z-w|^\alpha\cdot \mu)(a) 
=\int_{X\times X}
\frac{|z-w|\ d|\mu|(z,w)}{
|z-a|\cdot|w-a|}
$$
so we use the case
of Lemma \ref{L:Riesz-5} with $d=2$, $\mu=|\mu|$,
$s=t=1$ and $u=1$.
\end{proof}

Next, a uniform estimate for the Cauchy transforms
of all bounded multiples of a fixed measure $\mu$:
\begin{lemma}\Label{L:Cauchy-1}
Let the complex measure $\mu$
have compact support in $\C\times\C$ and no mass at the point
$(b,b)$.  Let $\delta>0$. Then there exists a
set $E$ having full area density at $b$ such that
$$ 
|a-b|^{1+\alpha}\cdot \left|\widehat{L(f\cdot\mu})(a)\right| 
< \delta\cdot \sup|f|
$$
for all $a\in E$, whenever $f:X\times X\to \C$
is a bounded Borel function.
\end{lemma}
\begin{proof} Take the set $E$ given by Lemma \ref{L:Newton-1}.
Then the desired estimate holds on $E$
since
$$
\left|\widehat{L(f\cdot\mu)}(a)\right|
\le \widetilde{H}(\mu)(a)\cdot\sup|f|
$$
for all $a$ for which the right-hand-side is finite.
\end{proof}
 
\subsection{The Product $g\cdot L$}
The dual of any Banach algebra is naturally
a module over the algebra.  In the present
situation, $\lip\alpha$ acts on $(\lip\alpha)^*$,
so given $g\in\lip\alpha$ and a measure $\mu$
on $X\times X$, 
we may define a new element $g\cdot L(\mu)$ of 
$(\lip\alpha)^*$ by setting
$$ \langle f,g\cdot L\rangle = \langle g\cdot f,L\rangle,
\ \forall f\in\lip\alpha.$$
We remark that $\langle1,g\cdot L\rangle = \langle g,L\rangle\not=0$,
in general, so we cannot represent $g\cdot L$ by a measure as in
Equation \eqref{E:2}.  However, writing
$$  f(z)g(z)- f(w)g(w)=
\left(  f(z)- f(w)\right)\cdot g(z) 
+  f(w)\cdot\left( g(z)-g(w)\right),$$
a short calculation gives
$$ \langle f,g\cdot L\rangle
= \int_{X\times X}
\frac{ f(z)- f(w)}{|z-w|^\alpha}
d\nu(z,w)
+
\int_X f(w)\,d\lambda(w),$$
where $\nu=(g\circ\pi_1)\cdot\mu$ is the measure on 
$X\times X$ such that
$$ \nu(E) = \int_E g(z)d\mu(z,w) $$
whenever $E\subset X\times X$ is a Borel set,
and $\lambda$ is the measure on $X$ such that
$$ \lambda(E) = \int_{E\times X}
\frac{g(z)-g(w)}{|z-w|^\alpha}d\mu(z,w)
$$
whenever $E\subset X$ is Borel, i.e. $\lambda$
is the first-coordinate marginal of the measure
$$  
\frac{g(z)-g(w)}{|z-w|^\alpha}\cdot\mu(z,w)
$$
(a {\em bounded} multiple of $\mu$).
So we may write $g\cdot L(\mu) = S_1+S_2$, where
$S_1=L(\nu)$, and
$S_2$ is represented by the measure $\lambda$ on $X$.

Let us call $S_1$ {\em the main part of } $g\cdot L$
and $S_2$ {\em the residual part of } $g\cdot L$.

We note for future reference that the measures
$\nu$ and $\lambda$ have total-variation measures
dominated by fixed measures depending only
on $\mu$ and on the $\lip\alpha$ norm
$\|g\|=\|g\|_\alpha  = \|g\|'+\sup|g|$
of $g$:  

\begin{lemma}\Label{L:domination}
Let $\mu$ be a measure on $X\times X$ having no mass on
the diagonal. Let $\mu_\sharp$ be the 
first-coordinate
marginal of $|\mu|$. 
Then for any $g\in\lip\alpha$, the 
measures $\nu$ and $\lambda$ representing the 
main and residual parts of $g\cdot L(\mu)$
satisfy\\
(1) $|\nu| \le |g(z)\cdot\mu|$, and
(2)
$|\lambda|\le \|g\|'\cdot|\mu_\sharp|$.
\end{lemma}
\qed

\subsection{Estimate for the Product}
We now establish a lemma that gives a set having full area 
density on which a uniform estimate holds
for $\widehat{g\cdot L(\mu)}$, provided $g(b)=0$.

\begin{lemma}\Label{L:Cauchy-2}
Let $0<\alpha<1$, and $b\in X$,  and let the complex measure $\mu$
have support in $X\times X$ and no mass on the diagonal. 
Then there exist a constant $K>0$
and a set $E\subset\C$ having full area
density at $b$, such that whenever
$g\in \tilde{A}$, with $g(b)=0$, we have 
$$ 
|a-b|\cdot \left|\widehat{g\cdot L(\mu)}(a)\right| \le
K\cdot\|g\|_\alpha
$$
for all $a\in E$.
\end{lemma}
\begin{proof} 
By Lemma \ref{L:avoid}, we may assume that $\mu$ has no mass
on the vertical slice $V=\{b\}\times X$, 
since we can if necessary replace it,
without altering $L(\mu)$,
by another measure $\mu'$ having no mass on $V$. 

Also, by Corollary \ref{C:Riesz-1} (taking $\delta=1$), 
there is a 
set $E_1$ having full density at $b$, such that
$$
|a-b|\cdot \int_{\R^d}
\frac{d|\mu_\sharp|(z)}{|z-a|} \le 1
$$
for all $a\in E_1$, where $\mu_\sharp$ is the usual
marginal (since $\mu_\sharp$ has no mass at $b$).

Applying Lemma \ref{L:Newton-1} with $\delta=1$, 
there is a set $E_2$
having full density at $b$, such that
$$ 
|a-b|^{1+\alpha}\cdot \left|\widehat{H}(\mu)(a)\right| 
\le 1
$$
for all $a\in E_2$.  

By Lemma \ref{L:Newton-2} there is a constant $M>0$
such that 
$$ 
|a-b|\cdot \left|\widetilde{H}(|z-w|^\alpha\cdot\mu)(a)\right| 
\le M
$$
for all $a\in E_3$, a set having full density at $b$.

Take $E=E_1\cap E_2\cap E_3$.  Then $E$ has full density at $b$.

Fix any $g\in\lip\alpha$, with $g(b)=0$ and $\|g\|\le1$,
and take $a\in E$.
Let $S_1=L(\nu)$ and $S_2=\lambda$ be the main and residual
parts of $g\cdot L(\mu)$. 
By Lemma \ref{L:domination}, $|\nu|\le|g(z)\cdot\mu|\le
|z-b|^\alpha\cdot|\mu|
\le 
|z-a|^\alpha\cdot|\mu|+
|a-b|^\alpha\cdot|\mu|
$, 
and $|\lambda|\le\mu_\sharp$, and we have
$$ \widehat{g\cdot L(\mu)} = 
\widehat{L(\nu)} + 
\widehat{\lambda},
$$
so for all $a\in E$
$$ \begin{array}{rcl}
&& 
\left| \widehat{g\cdot L(\mu)}(a) \right|
\\
&\le&
\dsty \widetilde{H}(\nu)(a)  + 
\int_\C \frac{d|\lambda|(z)}{|z-a|}
\\
&\le&
\dsty \widetilde{H}(|z-b|^\alpha\cdot|\mu|)(a) + 
|a-b|^\alpha\cdot\widetilde{H}(\mu)(a)
+
\int \frac{d|\mu_\sharp|(z)}{|z-a|}
\end{array}
$$
so
$$
|a-b|\cdot\left| \widehat{g\cdot L(\mu)}(a) \right|
\le M+2.
$$
Taking $K=M+2$, we are done.
\end{proof}

\section{Proof of Theorem}\Label{S:proof}

Suppose $A=A_\alpha(U)$ admits a nonzero
continuous point derivation at $b\in X=\bdy(U)$, and let
$\partial$ be the normalized derivation at that point.

There is a complex measure $\mu$ on $X\times X$
having no mass on the diagonal, such that
$ \partial f = L(\mu)(f)$ whenever
$f\in \tilde{A}$. 
(We express this by saying 
that the distribution $T_1:= L(\mu)$ is continuous
on $\lip\alpha$ and
{\em represents} the derivation on $\tilde{A}$.)
This fact was explained
in our previous paper
\cite[pp.~135-6, or p.~6 in the ArXiv copy]{bsaf}.

\smallskip
By Lemmas \ref{L:Cauchy-1} and \ref{L:Cauchy-2} 
there exist a constant
$K>0$ and a set $E$ having full area density at
$b$, such that for all $a\in E$ we have
$$
\begin{array}{rcl}
 |a-b|^{1+\alpha}\cdot\left|\widehat{T_1}(a)\right| 
&\le& 1,\\
|a-b|\cdot \widehat{g\cdot T_1}(a) 
&\le& K.
\end{array}
$$

Since $U$ has full density at $b$,
we may (and we do) assume $E\subset U$, by taking the intersection
if need be.

Defining distributions (and elements of $\lip\alpha^*$)
$T_0 := (z-b)\cdot T_1$, and 
$T := -\pi(z-b)^2\cdot T_1$, 
one sees that 
$T_0(f) = f(b)$ and $T(f)=0$ for all $f\in\mathcal{A}$,
and hence for all $f\in\tilde{A}$ by continuity.
In other words, $T_0$ {\em represents evaluation at $b$}
on $\tilde{A}$
and $T=-\pi(z-b)T_0$ {\em annihilates} $\tilde{A}$.

All these distributions have compact support (on $X$)
and so have Cauchy transforms, which are holomorphic
when restricted to $U$.  

Next (see \cite{bsaf} for details),
$$ \widehat{T_0} = (z-b)\cdot \widehat{T_1}, $$
and
$$ \widehat{T} = 1 - \pi(z-b)^2\cdot\widehat{T_1}. $$
Thus, for $a\in E$,
we have 
$$ \left| \widehat{T}(a) - 1\right|
\le  |a-b|^{1-\alpha}.
$$
(Note that the exponent $1-\alpha$ was mistakenly
entered as $\alpha$ in \cite[Equation (3), page 142]{bsaf}.
The ArXiv version is corrected.)
So, replacing $E$, if need be, by its intersection
with the open ball of radius $1$ about $b$,
 we may assume that $\widehat{T}(a)\not=0$ on $E$.

Then whenever $a\in E$ we may form the distribution 
$$R_a := \frac1{\pi\widehat{T}(a)(a-z)}\cdot T,$$
Then $R_a$ represents evaluation at $a$ on $A$, and the functional
$$ f\mapsto \frac{f(a)-f(b)}{a-b} - \partial f $$
is represented on $A$ by the distribution
$$
D_a = \frac{R_a-T_0}{a-b} - T_1.
$$ 
Then  $D_{a}f \to 0$ for all $f\in \mathcal{A}$ as
$a\to b$, with $a\in E$.
To prove the theorem, we have to show that
this also holds for all $f$ in the closure $\tilde{A}$ of $\mathcal A$.
To do this, it suffices to show that
the functionals $D_a$ are uniformly
bounded on $\tilde{A}$, for $a\in E$, i.e
that
$$ |D_a(f)|  \le c\|f\|_\alpha $$
for some constant $c>0$, for all $f\in A$ and all $a\in E$.

Fix an arbitrary $f\in\lip\alpha$, holomorphic on $U$,
with $\|f\|_\alpha\le1$.
Take $g(z)=f(z)-f(b)$, so $D_a(f)=D_a(g)$, $\|f\|'=\|g\|'$ and 
$g(b)=0$, so that $|g(a)|\le|a-b|^\alpha$.  
A calculation shows that
$$ D_a(g) = (a-b)\cdot\widehat{g\cdot T_1}(a) 
+ \pi(a-b)g(a)\widehat{T_1}(a).
$$
Thus, for $a\in E$, we have
$$ |D_a(g)| \le |a-b|\cdot|\widehat{g\cdot T_1}(a)| 
+ \pi|a-b|^{1+\alpha}\cdot|\widehat{T_1}(a)|\le K+\pi.
$$
Thus $D_a(f)$ is indeed bounded, as required.
This concludes the proof. 

\section{Remarks}
\subsection{}
The set $E$ constructed in the proof has
\lq\lq more than" full area density.  To be precise,
the complement of $E$
is a finite union of exceptional sets $E'$ for which one of the 
conditions
\begin{itemize}
\item
$\sum_n 2^nC_1(A_n(b)\cap E') <+\infty$,\\
\item
$\sum_n 2^{\alpha n}C_\alpha(A_n(b)\cap E') <+\infty$, or\\
\item
$\sum_n 2^{(1+\alpha)n}C_{1+\alpha}(A_n(b)\cap E') <+\infty$
\end{itemize}
holds. The weakest of these conditions is the latter, so we can say
that 
$$
\sum_n 2^{(1+\alpha)n}C_{1+\alpha}(A_n(b)\setminus E) <+\infty.
$$
The convergence of this Wiener-type series implies that
the $(1+\alpha)$-dimensional density of $C_{1+\alpha}$ capacity
$$ \lim_{r\downarrow0}
\frac{C_{1+\alpha}(\mathbb{B}(b,r)\setminus E)}{r^{1+\alpha}}=0,$$
and, \textit{a fortiori}, the $\beta$-dimensional density
of $\beta$-dimensional Hausdorff content
$$ \lim_{r\downarrow0}
\frac{M^{\beta}(\mathbb{B}(b,r)\setminus E)}{r^{\beta}}=0$$
for $1+\alpha<\beta\le2$.

\subsection{} It would be interesting to know whether
the conditions of Corollary \ref{C:main} are equivalent
to the existence of a sequence $a_n\to b$ along which
$f'(a_n)$ is bounded, for each $f\in A_\alpha(U)$.

\subsection{} I am grateful to David Malone and Oliver Mason,
and to the referees
for helpful comment on this paper.

\end{document}